\documentclass[10pt, reqno]{amsart}
\usepackage{amsmath, amsthm, amscd, amsfonts, amssymb, graphicx, color}
\usepackage[bookmarksnumbered, colorlinks, plainpages]{hyperref}
\hypersetup{colorlinks=true,linkcolor=red, anchorcolor=green, citecolor=cyan, urlcolor=red, filecolor=magenta, pdftoolbar=true}
\usepackage{mathrsfs}

\newtheorem{theorem}{Theorem}[section]
\newtheorem{lemma}[theorem]{Lemma}

\newtheorem{corollary}[theorem]{Corollary}
\theoremstyle{definition}
\newtheorem{definition}[theorem]{Definition}

\theoremstyle{remark}
\newtheorem{remark}[theorem]{Remark}
\numberwithin{equation}{section}

\begin{document}
\setcounter{page}{1}

\title[Number fields as curves over $\mathbf{F}_1$]{Number fields as curves over $\mathbf{F}_1$}

\author[Igor V. Nikolaev]
{Igor V. Nikolaev$^1$}

\address{$^{1}$ Department of Mathematics and Computer Science, St.~John's University, 8000 Utopia Parkway,  
New York,  NY 11439, United States.}
\email{\textcolor[rgb]{0.00,0.00,0.84}{igor.v.nikolaev@gmail.com}}

\dedicatory{All data are available as part of the manuscript}




\begin{abstract}
We study function fields in one variable over the field with one element $\mathbf{F}_1$.
It is proved that the Galois extensions of $\mathbf{Q}$ are isomorphic to the  curves over $\mathbf{F}_1$ being 
understood as the Deitmar schemes.  Specifically, one gets  explicit formulas linking the genus and the number of cusps 
of an algebraic curve over the extension  $\mathbf{F}_{1^m}$ of order  $m\ge 1$  of the  field $\mathbf{F}_1$  and  the $m$-th roots  of unity of  
the corresponding number field.  It follows that elliptic curves with one cusp over $\mathbf{F}_1$ are  either  the cyclotomic fields or
the maximal abelian unramified extensions of the quadratic number fields. 
Our  proof depends  on representaion of the Drinfeld modules by the bounded linear operators on a Hilbert space
and the crossed product structure of the Cuntz-Krieger algebras. 
\end{abstract}

\maketitle

\section{Introduction}
An  interplay  between  number fields and  the fields of rational functions 
in one variable is long known  [Dedekind \& Weber 1882] \cite{DedWeb1}. 
For instance, both fields are the Dedekind domains and their groups of units are finitely generated, if
 functions  are defined over the finite field $\mathbf{F}_{q}$.  
Likewise,  the class field theory is  essentially the same in  both cases  [Artin \& Whaples 1945]  \cite{ArtWap1}. 
Whether  or not  these relations  can be upgraded to a field  isomorphism  is an  interesting open problem.  

Let  $\mathbf{F}_1$ be the field with one element [Tits 1957] \cite[Section 13]{Tit1}.
Such an object is no longer  a field in the usual sense, since the axiom $0\ne 1$ fails for $\mathbf{F}_1$. 
Yet the idea of an ``absolute'' field  $\mathbf{F}_1$  is natural and powerful, see e.g. 
[Manin 1995] \cite[Section 1.6]{Man1},  [Kapranov \& Smirnov 1995] \cite{KapSmi1}, 
[Soul\'e 2004] \cite{Sou1},  [Deitmar 2005] \cite{Dei1}, [Connes, Consani \& Marcolli 2009] \cite{ConConMar1}, [Lorscheid 2018] \cite{Lor1} and others.
Using the Cuntz-Krieger algebras $\mathcal{O}_A$  [Cuntz \& Krieger 1980] \cite{CunKri1},
one can construct a covariant functor $V(\mathbf{F}_1)\to \mathcal{O}_A$,
where $V(\mathbf{F}_1)$ is the  projective variety over $\mathbf{F}_1$ \cite{Nik2}.

The aim of our note is an isomorphism between number fields and function fields in one variable over $\mathbf{F}_1$ (Theorem \ref{thm1.2}).
To formalize our results,   denote by $S_{g,n}$ the Riemann surface of genus $g\ge 0$ with $n\in \{1, 2, 3\}$ cusps. 
Let $\mathbb{A}(S_{g,n})$ and $\mathbb{A}_p(S_{g,n})$ be the cluster $C^*$-algebra
of $S_{g,n}$  and the congruence sub-algebra of $\mathbb{A}(S_{g,n})$ of level $p$, 
respectively \cite[Section 2.2.1]{Nik1}.  
Consider an isomorphism $\mathbb{A}_p(S_{g,n})/I_{\alpha}\cong \mathbb{A}_{RM}^{6g-6+2n}$,
where $\{I_{\alpha} ~|~\alpha\in\mathbf{R}^{6g-7+2n}\}$ is the maximal two-sided ideal of 
$\mathbb{A}_p(S_{g,n})$ and  $\mathbb{A}_{RM}^{6g-6+2n}$ is an enveloping 
$\operatorname{AF}$-algebra of the noncommutative torus   $\mathscr{A}_{RM}^{6g-6+2n}$
with real multiplication (RM) [Rieffel 1990] \cite{Rie1};   we refer the reader to  \cite[Section 2.2.2]{Nik1} 
for the notation and details. 
Let $\operatorname{Drin}_A^{r}(k)$ be the Drinfeld module of rank $r\ge 1$, where $A=\mathbf{F}_q[T]$ and 
$k= \mathbf{F}_q(T)$ [Rosen 2002] \cite[p. 200]{R}. 
Recall that: (i) there exists a functor $F: \operatorname{Drin}_A^{3g-3+n}(k)\to \mathscr{A}_{RM}^{6g-6+2n}$
from the category of Drinfeld modules to such of the noncommutative tori; 
(ii) $F(\Lambda_{\rho}[a])=\{e^{2\pi i\alpha_k+\log\log\varepsilon} ~|~1\le k\le 6g-7+2n\}$, 
where  $\Lambda_{\rho}[a]$ is the torsion submodule and $\varepsilon$ a unit in the number field generated by $\alpha_k$
and (iii) the Galois group  
 $G:=\operatorname{Gal} \left(\mathbf{k}(e^{2\pi i\alpha_k+\log\log\varepsilon})  ~| ~\mathbf{k}\right)\subseteq \operatorname{GL}_{3g-3+n}\left(A/aA\right)$,
where $\mathbf{k}$  is a subfield of  the number field $\mathbf{Q}(e^{2\pi i\alpha_k+\log\log \varepsilon} )$ \cite[Theorem 3.3]{Nik1}.
Moreover, the number field $\mathbf{K}\cong\mathbf{k}(e^{2\pi i\alpha_k+\log\log\varepsilon})$  ($\mathbf{K}\cong\mathbf{k}(\cos 2\pi\alpha_k\times\log\varepsilon)$, resp.)
is the Galois extension of $\mathbf{k}\subset\mathbf{C}-\mathbf{R}$  ($\mathbf{k}\subset\mathbf{R}$, resp.), such that  
$\operatorname{Gal}(\mathbf{K}|\mathbf{k})\cong G$ \cite[Corollary 3.4]{Nik1}. 
On the other hand, 
let  $\mathscr{K}$ be the $C^*$-algebra of compact operators on a Hilbert space and consider the stabilized 
Cuntz-Krieger algebra $\mathcal{O}_A\otimes\mathscr{K}$.  Roughly speaking, a link between number fields and function 
fields in one variable over $\mathbf{F}_1$ is derived from the well-known $C^*$-algebra  isomorphism:
\begin{equation}\label{eq1.1}
 \mathcal{O}_A\otimes\mathscr{K}\cong \mathbb{A}_{RM}^{6g-6+2n}\rtimes_{\sigma_A}\mathbf{Z},
 \end{equation}
where  $\mathbb{A}_V$ is a stationary $AF$-algebra given by a matrix $A\in \operatorname{GL}_{6g-6+2n}(\mathbf{Z})$ [Blackadar 1986]  \cite[Chapter 7]{B} 
and the crossed product $\rtimes$ is taken  by the shift automorphism $\sigma_A$ of $\mathbb{A}_{RM}^{6g-6+2n}$ [Blackadar 1986] \cite[Exercise 10.11.9 (b)]{B}.
Specifically, we identify the Riemann surface $S_{g,n}$ with the complex points of a curve 
in the affine space $\mathbf{A}^{6g-6+2n}$ given the system of polynomial equations
$\{ f_j(x_1,\dots, x_{6g-6+2n}) = 0 ~|~1\le j\le  6g-7+2n\}$. 
The norm closure of a self-adjoint representation of such  equations 
by the bounded linear operators on a Hilbert space generates  the two-sided ideal 
$I_{\alpha}\subset \mathbb{A}(S_{g,n})$ introduced earlier.  
Likewise, we  let $\mathbf{F}_{1^m}$ be an extension of order $m\ge 1$ of $\mathbf{F}_1$ and
$\mathbf{F}_{1^m}[x_1,\dots, x_{6g-6+2n}]$ be polynomials  over $\mathbf{F}_{1^m}$. 
Neither  $\mathbf{F}_{1^m}$ is a field nor $\mathbf{F}_{1^m}[x_1,\dots, x_{6g-6+2n}]$ is a ring
[Kapranov \& Smirnov 1995] \cite[Section 1.2]{KapSmi1} and [Soul\'e 2004] \cite[Section 2.4]{Sou1}.  
The set $\{\mathbf{F}_1[T] ~|~T=(x_1,\dots x_{6g-6+2n})\}$ 
has only the structure of a multiplicative monoid. 
However, one can take  a localization $T^{-1}$ in $T$ consisting of $0$ 
and all integer powers of $T$; this is  denoted by  $\mathbf{F}_1[T^{\pm 1}]$.
To obtain a ring from  $\mathbf{F}_1[T^{\pm 1}]$, one needs to take the base extension $\otimes_{\mathbf{F}_1}\mathbf{Z}$
to the ring of Laurent polynomials with the integer coefficients [Lorscheid 2018] \cite[Example 2.3]{Lor1}, i.e. 
$\mathbf{F}_1[T^{\pm 1}]\otimes_{\mathbf{F}_1}\mathbf{Z}\cong \mathbf{Z}[T^{\pm 1}]$. 
The  operation $\otimes_{\mathbf{F}_1}\mathbf{Z}$ is a functor sending morphisms between monoids 
to the ring homomorphisms [Deitmar 2005] \cite[Section 1]{Dei1};
hence the following notation. 
\begin{definition}\label{dfn1.1}
We shall denote by  $\mathbf{F}_{1}[x_1^{\pm 1},\dots, x_{6g-6+2n}^{\pm 1}]$
a monoid, such that: 
$\mathbf{F}_{1}[x_1^{\pm 1},\dots, x_{6g-6+2n}^{\pm 1}]\otimes_{\mathbf{F}_1}\mathbf{Z}\cong 
\mathcal{A}(\mathbf{x},  S_{g,n})
\subseteq \mathbf{Z}[x_1^{\pm 1},\dots, x_{6g-6+2n}^{\pm 1}]$.
Likewise, we  use symbol $[S_{g,n}]$ for the maximal ideal [Deitmar 2005] \cite[Section 1.2]{Dei1} of 
\linebreak
$\mathbf{F}_{1}[x_1^{\pm 1},\dots, x_{6g-6+2n}^{\pm 1}]$ obtained from
the inclusion of rings  $K_0(I_{\alpha})\subset \mathcal{A}(\mathbf{x},  S_{g,n})$. 
The ideal  $[S_{g,n}]$ is called palindromic (apalindromic, resp.),  
if $[S_{g,n}]\otimes_{\mathbf{F}_1}\mathbf{Z}$  consists of the Laurent polynomials  $p(x_1^{\mp 1},\dots, x_{6g-6+2n}^{\mp 1})=p(x_1^{\pm 1},\dots, x_{6g-6+2n}^{\pm 1})$
\linebreak
($p(x_1^{\mp 1},\dots, x_{6g-6+2n}^{\mp 1}) \ne p(x_1^{\pm 1},\dots, x_{6g-6+2n}^{\pm 1})$, resp.).
By a function field in one variable over $\mathbf{F}_1$ one understands 
the quotient field
\linebreak
 $\mathbf{F}_{1}[x_1^{\pm 1},\dots, x^{\pm 1}_{6g-6+2n}]\otimes_{\mathbf{F}_1}\mathbf{Z} ~/~[S_{g,n}]\otimes_{\mathbf{F}_1}\mathbf{Z}$. 
\end{definition}
Our main results can be formulated as follows. 
\begin{theorem}\label{thm1.2}
For each $m\ge 1$, the following fields are isomorphic:
\begin{equation} 
\begin{array}{lcr}
 \mathbf{F}_{1^m}[x_1^{\pm 1},\dots, x^{\pm 1}_{6g-6+2n}] \otimes_{\mathbf{F}_1}\mathbf{Z}~/~[S_{g,n}]\otimes_{\mathbf{F}_1}\mathbf{Z}
\cong&&\\
&&\\
\cong
\begin{cases} \mathbf{Q}\left(e^{2\pi i\alpha_k +\log\log\varepsilon} ~|~1\le k\le 6g-7+2n \right)
(e^{\frac{2\pi i}{m}}), & if ~[S_{g,n}] ~is\cr
& apalindromic,\cr
               \mathbf{Q}\left(\cos 2\pi(\alpha_k +\frac{1}{m})\times\log\varepsilon ~|~1\le k\le 6g-7+2n\right), 
               &  if ~[S_{g,n}] ~is\cr
& palindromic.
\nonumber
\end{cases}               
&&
\end{array}
\end{equation}
\end{theorem}

\medskip
\begin{remark}
Theorem \ref{thm1.2} says  that the fields  $\mathbf{F}_{1^m}$ are one-to-one with  the  
roots  of unity of order $m\ge 1$;  this was conjectured in   [Kapranov \& Smirnov 1995] \cite[Section 1.2]{KapSmi1}.
\end{remark}

\medskip
\begin{definition}
The number fields 
$\mathbf{Q}\left(e^{2\pi i\alpha_k +\log\log\varepsilon} ~|~1\le k\le 6g-7+2n \right)$ and 
 $\mathbf{Q}\left(\cos 2\pi\alpha_k \times\log\varepsilon ~|~1\le k\le 6g-7+2n\right)$
 are said to be of the topological type $(g,n)$. 
 \end{definition}

\begin{corollary}\label{cor1.5}
 The number fields of topological type $(1,1)$  are  abelian extensions
 of the form:
\begin{equation}\label{eq1.2}
\begin{array}{lcr}
 \mathbf{K} \cong
\begin{cases} \mathbf{Q}\left(e^{2\pi i\alpha +\log\log\varepsilon}\right)
(e^{\frac{2\pi i}{m}}), & if  ~\mathbf{K}\subset\mathbf{C}-\mathbf{R} \cr
               \mathbf{Q}\left(\cos 2\pi(\alpha +\frac{1}{m})\times\log\varepsilon\right), 
               &  if ~\mathbf{K}\subset\mathbf{R}.
\end{cases}               
\end{array}
\end{equation}
  \end{corollary}

\medskip
The paper is organized as follows.  A brief review of the preliminary facts is 
given in Section 2. Theorem \ref{thm1.2} and Corollary \ref{cor1.5}
are   proved in Section 3.  The number fields of topological type $(1,1)$ 
are  considered in Section 4.

\section{Preliminaries}

\subsection{$C^*$-algebras}
The $C^*$-algebra is an algebra  $\mathscr{A}$ over $\mathbf{C}$ with a norm 
$a\mapsto ||a||$ and an involution $\{a\mapsto a^* ~|~ a\in \mathscr{A}\}$  such that $\mathscr{A}$ is
complete with  respect to the norm, and such that $||ab||\le ||a||~||b||$ and $||a^*a||=||a||^2$ for every  $a,b\in \mathscr{A}$.  
Each commutative $C^*$-algebra is  isomorphic
to the algebra $C_0(X)$ of continuous complex-valued
functions on some locally compact Hausdorff space $X$. 
Any other  algebra $\mathscr{A}$ can be thought of as  a noncommutative  
topological space. 

\subsubsection{K-theory of $C^*$-algebras}
By $M_{\infty}(\mathscr{A})$ 
one understands the algebraic direct limit of the $C^*$-algebras 
$M_n(\mathscr{A})$ under the embeddings $a\mapsto ~\mathbf{diag} (a,0)$. 
The direct limit $M_{\infty}(\mathscr{A})$  can be thought of as the $C^*$-algebra 
of infinite-dimensional matrices whose entries are all zero except for a finite number of the
non-zero entries taken from the $C^*$-algebra $\mathscr{A}$.
Two projections $p,q\in M_{\infty}(\mathscr{A})$ are equivalent, if there exists 
an element $v\in M_{\infty}(\mathscr{A})$,  such that $p=v^*v$ and $q=vv^*$. 
The equivalence class of projection $p$ is denoted by $[p]$.   
We write $V(\mathscr{A})$ to denote all equivalence classes of 
projections in the $C^*$-algebra $M_{\infty}(\mathscr{A})$, i.e.
$V(\mathscr{A}):=\{[p] ~:~ p=p^*=p^2\in M_{\infty}(\mathscr{A})\}$. 
The set $V(\mathscr{A})$ has the natural structure of an abelian 
semi-group with the addition operation defined by the formula 
$[p]+[q]:=\mathbf{diag}(p,q)=[p'\oplus q']$, where $p'\sim p, ~q'\sim q$ 
and $p'\perp q'$.  The identity of the semi-group $V(\mathscr{A})$ 
is given by $[0]$, where $0$ is the zero projection. 
By the $K_0$-group $K_0(\mathscr{A})$ of the unital $C^*$-algebra $\mathscr{A}$
one understands the Grothendieck group of the abelian semi-group
$V(\mathscr{A})$, i.e. a completion of $V(\mathscr{A})$ by the formal elements
$[p]-[q]$.  The image of $V(\mathscr{A})$ in  $K_0(\mathscr{A})$ 
is a positive cone $K_0^+(\mathscr{A})$ defining  the order structure $\le$  on the  
abelian group  $K_0(\mathscr{A})$. The pair   $\left(K_0(\mathscr{A}),  K_0^+(\mathscr{A})\right)$
is known as a dimension group of the $C^*$-algebra $\mathscr{A}$.

\subsubsection{$\operatorname{AF}$-algebras}
An {\it AF-algebra}  (Approximately Finite-dimensional $C^*$-algebra) is defined to
be the  norm closure of an ascending sequence of   finite dimensional
$C^*$-algebras $M_n$,  where  $M_n$ is the $C^*$-algebra of the $n\times n$ matrices
with entries in $\mathbf{C}$. Here the index $n=(n_1,\dots,n_k)$ represents
the  semi-simple matrix algebra  $M_i=M_{i_1}\oplus\dots\oplus M_{i_k}$.
The ascending sequence mentioned above  can be written as 
$M_{i_1}\buildrel\rm\varphi_1\over\longrightarrow M_{i_2}
   \buildrel\rm\varphi_2\over\longrightarrow\dots$,
where $M_{i_k}$ are the finite dimensional $C^*$-algebras and
$\varphi_i$ the homomorphisms between such algebras.  
If $\varphi_i=Const$, then the AF-algebra $\mathscr{A}$ is called 
{\it stationary}. 
The homomorphisms $\varphi_i$ can be arranged into  a graph as follows. 
Let  $M_i=M_{i_1}\oplus\dots\oplus M_{i_k}$ and 
$M_{i'}=M_{i_1'}\oplus\dots\oplus M_{i_k'}$ be 
the semi-simple $C^*$-algebras and $\varphi_i: M_i\to M_{i'}$ the  homomorphism. 
One has  two sets of vertices $V_{i_1},\dots, V_{i_k}$ and $V_{i_1'},\dots, V_{i_k'}$
joined by  $a_{rs}$ edges  whenever the summand $M_{i_r}$ contains $a_{rs}$
copies of the summand $M_{i_s'}$ under the embedding $\varphi_i$. 
As $i$ varies, one obtains an infinite graph called the   Bratteli diagram of the
AF-algebra.  The matrix $A=(a_{rs})$ is known as  a  partial multiplicity matrix;
an infinite sequence of $A_i$ defines a unique AF-algebra.
If   $\mathbb{A}$ is a stationary AF-algebra, then   $A_i=Const$
for all $i\ge 1$.  
The  dimension group $\left(K_0(\mathbb{A}),  K_0^+(\mathbb{A})\right)$  is a complete invariant of the
Morita equivalence class of the AF-algebra $\mathbb{A}$.

\subsection{Cluster $C^*$-algebras}
The cluster algebra  of rank $n$ 
is a subring  $\mathcal{A}(\mathbf{x}, B)$  of the field  of  rational functions in $n$ variables
depending  on  variables  $\mathbf{x}=(x_1,\dots, x_n)$
and a skew-symmetric matrix  $B=(b_{ij})\in M_n(\mathbf{Z})$.
The pair  $(\mathbf{x}, B)$ is called a  seed.
A new cluster $\mathbf{x}'=(x_1,\dots,x_k',\dots,  x_n)$ and a new
skew-symmetric matrix $B'=(b_{ij}')$ is obtained from 
$(\mathbf{x}, B)$ by the   exchange relations [Williams 2014]  \cite[Definition 2.22]{Wil1}:
\begin{eqnarray}
x_kx_k'  &=& \prod_{i=1}^n  x_i^{\max(b_{ik}, 0)} + \prod_{i=1}^n  x_i^{\max(-b_{ik}, 0)},\cr \nonumber
b_{ij}' &=& 
\begin{cases}
-b_{ij}  & \mbox{if}   ~i=k  ~\mbox{or}  ~j=k\cr
b_{ij}+{|b_{ik}|b_{kj}+b_{ik}|b_{kj}|\over 2}  & \mbox{otherwise.}
\end{cases}
\end{eqnarray}
The seed $(\mathbf{x}', B')$ is said to be a  mutation of $(\mathbf{x}, B)$ in direction $k$.
where $1\le k\le n$.  The  algebra  $\mathcal{A}(\mathbf{x}, B)$ is  generated by the 
cluster  variables $\{x_i\}_{i=1}^{\infty}$
obtained from the initial seed $(\mathbf{x}, B)$ by the iteration of mutations  in all possible
directions $k$.   The  Laurent phenomenon
 says  that  $\mathcal{A}(\mathbf{x}, B)\subset \mathbf{Z}[\mathbf{x}^{\pm 1}]$,
where  $\mathbf{Z}[\mathbf{x}^{\pm 1}]$ is the ring of  the Laurent polynomials in  variables $\mathbf{x}=(x_1,\dots,x_n)$
 [Williams 2014]  \cite[Theorem 2.27]{Wil1}.
In particular, each  generator $x_i$  of  the algebra $\mathcal{A}(\mathbf{x}, B)$  can be 
written as a  Laurent polynomial in $n$ variables with the   integer coefficients.

 The cluster algebra  $\mathcal{A}(\mathbf{x}, B)$  has the structure of an additive abelian
semigroup consisting of the Laurent polynomials with positive coefficients. 
In other words,  the $\mathcal{A}(\mathbf{x}, B)$ is a dimension group
[Blackadar 1986] \cite[Section 7.3]{B}.
The cluster $C^*$-algebra  $\mathbb{A}(\mathbf{x}, B)$  is   an  AF-algebra,  
such that $K_0(\mathbb{A}(\mathbf{x}, B))\cong  \mathcal{A}(\mathbf{x}, B)$.

\subsubsection{Cluster $C^*$-algebra $\mathbb{A}(S_{g,n})$}
Denote by $S_{g,n}$  the Riemann surface   of genus $g\ge 0$  with  $n\ge 0$ cusps.
 Let   $\mathcal{A}(\mathbf{x},  S_{g,n})$ be the cluster algebra 
 coming from  a triangulation of the surface $S_{g,n}$   [Williams 2014]  \cite[Section 3.3]{Wil1}. 
 We shall denote by  $\mathbb{A}(S_{g,n})$  the corresponding cluster $C^*$-algebra. 
 Let $p$ be a prime number, and denote by  $\mathcal{A}_p(S_{g,n})$ a sub-algebra of  $\mathcal{A}(S_{g,n})$
consisting of the Laurent polynomials whose coefficients are divisible by $p$. It is easy to verify that   $\mathcal{A}_p(S_{g,n})$
is again a dimension group under the addition of the Laurent polynomials. We say that  $\mathbb{A}_p(S_{g,n})$ is a
congruence sub-algebra of level $p$ of the cluster $C^*$-algebra  $\mathbb{A}(S_{g,n})$, i.e. 
$K_0(\mathbb{A}_p(S_{g,n}))\cong \mathcal{A}_p(S_{g,n})$.

Let $T_{g,n}$ be the Teichm\"uller space of the surface $S_{g,n}$,
i.e. the set of all complex structures on $S_{g,n}$ endowed with the 
natural topology. The geodesic flow $T^t: T_{g,n}\to T_{g,n}$
is a one-parameter  group of matrices $\operatorname{diag} \left( e^t, e^{-t}\right)$
acting on the holomorphic quadratic differentials on the Riemann surface $S_{g,n}$. 
$\sigma_t: \mathbb{A}(S_{g,n})\to \mathbb{A}(S_{g,n})$
called the Tomita-Takesaki flow on the AF-algebra $\mathbb{A}(S_{g,n})$. 
Denote by $Prim~\mathbb{A}(S_{g,n})$ the space of all primitive ideals 
of $\mathbb{A}(S_{g,n})$ endowed with the Jacobson topology. 
Recall (\cite{Nik2}) that each primitive ideal has a parametrization by a vector 
$\Theta\in \mathbf{R}^{6g-7+2n}$ and we write it 
$I_{\Theta}\in Prim~\mathbb{A}(S_{g,n})$.
\begin{theorem}\label{thm2.2}
{\bf (\cite{Nik0})}
There exists a homeomorphism
$h:  Prim~\mathbb{A}(S_{g,n})\times \mathbf{R}\to \{U\subseteq  T_{g,n} ~|~U~\hbox{{\sf is generic}}\},$
where $h(I_{\Theta},t)=S_{g,n}$ is 
given by the formula $\sigma_t(I_{\Theta})\mapsto S_{g,n}$;  the set $U=T_{g,n}$ if and only if
$g=n=1$.   The $\sigma_t(I_{\Theta})$
is an ideal of  $\mathbb{A}(S_{g,n})$ for all $t\in \mathbf{R}$ and 
 the quotient  algebra $AF$-algebra  $\mathbb{A}(S_{g,n})/\sigma_t(I_{\Theta}):=\mathbb{A}_{\Theta}^{6g-6+2n}$
is  a non-commutative coordinate ring  of  the Riemann surface  $S_{g,n}$.  
\end{theorem}

\subsection{Noncommutative tori}
The $m$-dimensional noncommutative torus $\mathscr{A}_{\Theta}^m$ is a
universal $C^*$-algebra  generated by the unitary operators $u_1,\dots, u_m$
satisfying the commutation relations 
$\{u_ju_i=e^{2\pi i \theta_{ij}} u_iu_j ~|~ 1\le i,j\le m\}$
for a skew-symmetric matrix  $\Theta=(\theta_{ij})\in M_m(\mathbf{R})$
[Rieffel 1990] \cite{Rie1}. 
 It is known that 
  $K_0(\mathscr{A}_{\Theta}^m)\cong K_1(\mathscr{A}_{\Theta}^m)\cong \mathbf{Z}^{2^{m-1}}$.
The canonical trace $\tau$ on the $C^*$-algebra
$\mathscr{A}_{\Theta}^m$ defines a homomorphism from 
$K_0(\mathscr{A}_{\Theta}^m)$ to the real line $\mathbf{R}$;
under the homomorphism, the image of $K_0(\mathscr{A}_{\Theta}^m)$
is a $\mathbf{Z}$-module, whose generators $\tau=(\tau_i)$ are polynomials 
in $\theta_{ij}$.

\subsubsection{Relation to the $\operatorname{AF}$-algebra  $\mathbb{A}_{\Theta}^{6g-6+2n}$}
The  $\mathscr{A}_{\Theta}^{6g-6+2n}$
is a crossed product $C^*$-algebra embedded into an $AF$-algebra 
$\mathbb{A}_{\Theta}^{6g-6+2n}$,  such that 
$K_0(\mathscr{A}_{\Theta}^{6g-6+2n})\cong K_0( \mathbb{A}_{\Theta}^{6g-6+2n})$.
Therefore  matrix $\Theta$ takes  the form:
\begin{equation}\label{eq2.1}
\Theta=
\left(
\begin{matrix}
0 & \alpha_1 &  &\cr
-\alpha_1 & 0 & \alpha_2 & \cr
  & -\alpha_2 & 0  &  &\cr
  \vdots & & & \vdots\cr
             & & & 0 & \alpha_{6g-7+2n} \cr
             & &  & - \alpha_{6g-7+2n} &0
\end{matrix}
\right). 
\end{equation}
The noncommutative torus $\mathscr{A}_{\Theta}^{6g-6+2n}$  is said to have real multiplication,   
if all $\alpha_k$ in  (\ref{eq2.1})  are algebraic numbers; we write $\mathscr{A}_{RM}^{6g-6+2n}$ in this case. 
Likewise,  one can think of $\alpha_k$ as components of 
 a normalized eigenvector $(1,\alpha_1,\dots,\alpha_{6g-7+2n})$ corresponding to the 
 Perron-Frobenius eigenvalue $\varepsilon>1$ of a positive matrix 
 $B\in GL_{6g-6+2n}(\mathbf{Z})$.

\subsection{Cuntz-Krieger algebras}
 The Cuntz-Krieger algebra is a  $C^*$-algebra $\mathcal{O}_A$
generated by the  partial isometries $s_1,\dots, s_m$ which satisfy  the relations:
\begin{equation}
\left\{
\begin{array}{ccc}
s_1^*s_1 &=& a_{11} s_1s_1^*+a_{12} s_2s_2^*+\dots+a_{1m}s_ms_m^*\\ 
s_2^*s_2 &=& a_{21} s_1s_1^*+a_{22} s_2s_2^*+\dots+a_{2m}s_ms_m^*\\ 
                  &\dots&\\
s_m^*s_m &=& a_{m1} s_1s_1^*+a_{m2} s_2s_2^*+\dots+a_{mm}s_ms_m^*,             
\end{array}
\right.
\end{equation}
where $A=(a_{ij})$ is a square matrix with  $a_{ij}\in \{0, 1, 2, \dots \}$ 
[Cuntz \& Krieger 1980] \cite{CunKri1}.
If  $\mathbb{A}$ be a stationary  $AF$-algebra given by matrix $A$
and $\mathscr{K}$ is the $C^*$-algebra of compact operators, 
then
\begin{equation}
\mathcal{O}_A\otimes \mathscr{K}\cong \mathbb{A}\rtimes_{\sigma_A}\mathbf{Z}, 
\end{equation}
where the crossed product is taken by the shift automorphism $\sigma_A$
of the Bratteli diagram of $\mathbb{A}$  \cite[Exercise 10.11.9 (b)]{B}. 
The $AF$-algebra $\mathbb{A}$ is a subalgebra of the stabilized Cuntz-Krieger algebra  $\mathcal{O}_A\otimes\mathscr{K}$,
i.e. there exists an injective homomorphism (embedding)  $\mathbb{A}\to\mathcal{O}_A\otimes\mathscr{K}$
[Cuntz \& Krieger 1980] \cite{CunKri1}.

\section{Proof}
\subsection{Proof of Theorem \ref{thm1.2}}
For the sake of clarity, let us outline the main ideas. 
Roughly speaking, Theorem \ref{thm1.2} follows from an isomorphism of the $C^*$-algebras 
given by formula (\ref{eq1.1}).  Namely,  let $S_{g,n}(\mathbf{F}_1)$ be a curve over $\mathbf{F}_1$
and $S_{g,n}(\mathbf{F}_1)\to \mathcal{O}_A$ the covariant functor  to the Cuntz-Krieger algebra.  
 The isomorphism (\ref{eq1.1}) says that the stabilized  $\mathcal{O}_A$ is generated by the $\operatorname{AF}$-algebra 
 $\mathbb{A}_{RM}^{6g-6+2n}$
 with an  embedded copy of the noncommutative torus  $\mathscr{A}_{RM}^{6g-6+2n}$ (Section 2.3.1).  
 It is known that the Grothendieck semigroup $K_0^+(\mathscr{A}_{RM}^{6g-6+2n})\cong \mathbf{Z}+\mathbf{Z}\alpha_1+
 \dots+\mathbf{Z}\alpha_{6g-7+2n}$ and its exponent $\exp [K_0^+(\mathscr{A}_{RM}^{6g-6+2n})]\cong \mathbf{Z}+\mathbf{Z}e^{2\pi i\alpha_1+\log\log\varepsilon}+
 \dots+\mathbf{Z}e^{2\pi i\alpha_{6g-7+2n}+\log\log\varepsilon}$ generates a Galois extension of $\mathbf{Q}$ \cite[Theorem 3.3]{Nik1}. 
 On the other hand, $\mathbb{A}_{RM}^{6g-6+2n}\cong \mathbb{A}(S_{g,n})/I_{\alpha}$
 and $K_0(\mathbb{A}(S_{g,n}))\cong \mathcal{A}(\mathbf{x},  S_{g,n})\subseteq \mathbf{Z}[x_1^{\pm 1},\dots, x_{6g-6+2n}^{\pm 1}]$;
 hence    one gets an isomorphism between  $\exp ~[K_0^+(\mathscr{A}_{RM}^{6g-6+2n})]\otimes\mathbf{Q}$ and the function field in one
 variable 
  $\mathbf{F}_{1}[x_1^{\pm 1},\dots, x^{\pm 1}_{6g-6+2n}]\otimes_{\mathbf{F}_1}\mathbf{Z} ~/~[S_{g,n}]\otimes_{\mathbf{F}_1}\mathbf{Z}$.
 (Definition \ref{dfn1.1}).
We pass to a detailed argument by splitting the proof in a series of lemmas. 

\begin{lemma}\label{lm3.1}
The map $S_{g,n}(\mathbf{F}_1)\to \mathcal{O}_A$ \cite{Nik2} gives rise to a functor 
$S_{g,n}(\mathbf{F}_1)\to \mathscr{A}_{RM}^{6g-6+2n}$, such that:

\medskip
(i) $\mathscr{A}_{RM}^{6g-6+2n}\subset \mathbb{A}_{RM}^{6g-6+2n}$, where the 
$\mathbb{A}_{RM}^{6g-6+2n}$ is a stationary $\operatorname{AF}$-algebra
corresponding to  the matrix $A$;

\smallskip
(ii)   the number field   $\exp ~[K_0^+(\mathscr{A}_{RM}^{6g-6+2n})]:=\mathbf{Q}\left(e^{2\pi i\alpha_k+\log\log\varepsilon}\right)$,
where $1\le k\le 6g-7+2n$, $\varepsilon$ is the eigenvalue and $(1, \alpha_1,\dots,\alpha_{6g-7+2n})$ is the normalized eigenvector
of the matrix $A$, is an invariant of morphisms of $S_{g,n}(\mathbf{F}_1)$. 
\end{lemma} 
\begin{proof}
(i) The functor $S_{g,n}(\mathbf{F}_1)\to \mathscr{A}_{RM}^{6g-6+2n}$ is defined as follows. 
 Let $S_{g,n}(\mathbf{F}_1)\to \mathcal{O}_A$ be the natural map constructed in \cite{Nik2}. 
Clearly, such a map extends to the stabilized Cuntz-Krieger algebras $\mathcal{O}_A\otimes\mathscr{K}$,
i.e. one gets a natural map  $S_{g,n}(\mathbf{F}_1)\to \mathcal{O}_A\otimes\mathscr{K}$. 
On the other hand, the isomorphism (\ref{eq1.1}) implies a natural map:
\begin{equation}\label{eq3.1}
S_{g,n}(\mathbf{F}_1)\to \mathbb{A}_{RM}^{6g-6+2n}\rtimes_{\sigma_A}\mathbf{Z},
\end{equation}
where $\sigma_A$ is the shift automorphism of the stationary $\operatorname{AF}$-algebra 
 $\mathbb{A}_{RM}^{6g-6+2n}$ given by the matrix $A$  [Blackadar 1986] \cite[Exercise 10.11.9 (b)]{B}.
Since both the $\mathbb{A}_{RM}^{6g-6+2n}$ and $\sigma_A$ depend on the single matrix $A$,
we can restrict the natural map (\ref{eq3.1}) to the $\operatorname{AF}$-algebra $\mathbb{A}_{RM}^{6g-6+2n}$.
In other words, one gets a functor  $S_{g,n}(\mathbf{F}_1)\to \mathbb{A}_{RM}^{6g-6+2n}$. 
The noncommutative torus  $\mathscr{A}_{RM}^{6g-6+2n}$ is densely embedded into the 
$\operatorname{AF}$-algebra  $\mathbb{A}_{RM}^{6g-6+2n}$ (Section 2.3.1) and thus one gets 
a functor $S_{g,n}(\mathbf{F}_1)\to \mathscr{A}_{RM}^{6g-6+2n}$.  
Item (i) of Lemma \ref{lm3.1} is proved.

\medskip
(ii) Recall that the number field $\mathbf{Q}\left(e^{2\pi i\alpha_k+\log\log\varepsilon} ~|~1\le k\le 6g-7+2n\right)$
is generated by $F(\Lambda_{\rho}[a])$, where $F: \operatorname{Drin}_A^{3g-3+n}(k)\to \mathscr{A}_{RM}^{6g-6+2n}$
is the functor from the category of Drinfeld modules to such of the noncommutative tori
and  $\Lambda_{\rho}[a]$ is the torsion submodule \cite[Theorem 3.3]{Nik1}.
Such a number field is an invariant of the morphisms of the modules  $\operatorname{Drin}_A^{3g-3+n}(k)$, 
{\it ibid}. On the other hand, the natural map $S_{g,n}(\mathbf{F}_1)\to \mathscr{A}_{RM}^{6g-6+2n}$
preserves the arrows (morphisms) of the respective categories. We conclude therefore that the 
number field  $\exp ~[K_0^+(\mathscr{A}_{RM}^{6g-6+2n})]$
is an invariant  of $S_{g,n}(\mathbf{F}_1)$. 

\medskip
Lemma \ref{lm3.1} is proved. 
\end{proof}

\begin{lemma}\label{lm3.2}
The function field 
 $\mathbf{F}_{1}[x_1^{\pm 1},\dots, x^{\pm 1}_{6g-6+2n}]\otimes_{\mathbf{F}_1}\mathbf{Z} ~/~[S_{g,n}]\otimes_{\mathbf{F}_1}\mathbf{Z}$
is isomorphic to  $\exp ~[K_0^+(\mathscr{A}_{RM}^{6g-6+2n})]\otimes\mathbf{Q}$
($\mathfrak{Re}~\exp ~[K_0^+(\mathscr{A}_{RM}^{6g-6+2n})]\otimes\mathbf{Q}$, resp.)
if $[S_{g,n}]$ is apalindromic (palindromic, resp.). 
\end{lemma} 
\begin{proof}
(i) The enveloping $\operatorname{AF}$-algebra 
 $\mathbb{A}_{RM}^{6g-6+2n}\cong \mathbb{A}(S_{g,n})/I_{\alpha}$
 gives rise to    a commutative  diagram in Figure 1.
The vertical arrows $K_0$ and $\otimes_{\mathbf{F}_1}\mathbf{Z}$ on the diagram  are the 
$K_0$-functor [Blackadar 1986] \cite{B}  and  the extension of base functor $\otimes_{\mathbf{F}_1}\mathbf{Z}$ [Deitmar 2005] \cite{Dei1}, 
respectively.

\begin{figure}[h]
\begin{picture}(300,150)(0,0)


\put(20,130){\vector(0,-1){35}}
\put(130,130){\vector(0,-1){35}}
\put(45,143){\vector(1,0){35}}
\put(43,143){$\subset$}

\put(25,115){$\scriptstyle K_0$}
\put(135,115){$\scriptstyle K_0$}

\put(25,55){$\scriptstyle \otimes_{\mathbf{F}_1}\mathbf{Z}$}
\put(135,55){$\scriptstyle \otimes_{\mathbf{F}_1}\mathbf{Z}$}

\put(17,140){$\scriptstyle I_{\alpha}$}
\put(120,140){$\scriptstyle \mathbb{A}(S_{g,n})$}


\put(20,70){\vector(0,-1){35}}
\put(130,70){\vector(0,-1){35}}
\put(45,23){\vector(1,0){35}}
\put(43,23){$\subset$}
\put(45,83){\vector(1,0){35}}
\put(43,83){$\subset$}

\put(10,20){$ \scriptstyle [S_{g,n}]$}

\put(10,80){$\scriptstyle K_0(I_{\alpha})$}

\put(92,80){$\scriptstyle\mathbf{Z}[x_1^{\pm 1},\dots,x_{6g-6+2n}^{\pm 1}]$}



\put(240,130){\vector(0,-1){35}}
\put(175,143){\vector(1,0){20}}

\put(232,140){$\scriptstyle \mathbb{A}_{RM}^{6g-6+2n}$}


\put(240,70){\vector(0,-1){35}}
\put(175,23){\vector(1,0){20}}
\put(175,83){\vector(1,0){20}}

\put(92,20){$\scriptstyle\mathbf{F}_1[x_1^{\pm 1},\dots,x_{6g-6+2n}^{\pm 1}]$}
\put(202,20){$\scriptstyle\mathbf{F}_1[x_1^{\pm 1},\dots,x_{6g-6+2n}^{\pm 1}]/[S_{g,n}]$}

\put(202,80){$\scriptstyle\mathbf{Z}[x_1^{\pm 1},\dots,x_{6g-6+2n}^{\pm 1}]/K_0(I_{\alpha})$}

\end{picture}
\caption{Reduction diagram}
\end{figure}

\medskip
(ii) If $\mathbf{K}$ be a Galois extension of the field $\mathbf{Q}$,
then $\mathbf{K}\cong\exp ~[K_0^+(\mathscr{A}_{RM}^{6g-6+2n})]\otimes\mathbf{Q}$   ($\mathbf{K}\cong \mathfrak{Re}~\exp ~[K_0^+(\mathscr{A}_{RM}^{6g-6+2n})]\otimes\mathbf{Q}$, resp.),
if  $\mathbf{K}$ is totally imaginary (totally real, resp.)  \cite[Corollary 3.4]{Nik1}.  
The number field $\mathbf{K}$ is invariant  of $S_{g,n}(\mathbf{F}_1)$, see Lemma \ref{lm3.1} (ii).

\medskip
(iii)  Let $[S_{g,n}]$ be palindromic,  i.e.   $p(x_1^{\mp 1},\dots, x_{6g-6+2n}^{\mp 1})=p(x_1^{\pm 1},\dots, x_{6g-6+2n}^{\pm 1})$; see  Definition \ref{dfn1.1}. 
Such a property means that the Laurent polynomilals belonging to  the ideal $[S_{g,n}]\otimes_{\mathbf{F}_1}\mathbf{Z}$ are invariant of the involution 
$\theta:  x_j^{\pm 1}\mapsto x_j^{\mp 1}$.  The pullback of $\theta$ 
is   complex conjugation of the algebra
 $\mathbb{A}_{RM}^{6g-6+2n}\cong  \mathfrak{Re} ~\mathbb{A}_{RM}^{6g-6+2n}
+i  ~\mathfrak{Im}  ~\mathbb{A}_{RM}^{6g-6+2n}.$ 
In other words,  the real $C^*$-algebra  $\mathfrak{Re} ~\mathbb{A}_{RM}^{6g-6+2n}$ is preserved by the involution $\theta$ \cite[Section 4.2]{Nik1}. 
Since $K_0^+(\mathbb{A}_{RM}^{6g-6+2n})\cong K_0^+(\mathscr{A}_{RM}^{6g-6+2n})$,
we conclude that the number field 
$\mathfrak{Re}~\exp ~[K_0^+(\mathscr{A}_{RM}^{6g-6+2n})]\otimes\mathbf{Q}$
is isomorphic to the function field 
$\mathbf{F}_{1}[x_1^{\pm 1},\dots, x^{\pm 1}_{6g-6+2n}]\otimes_{\mathbf{F}_1}\mathbf{Z} ~/~[S_{g,n}]\otimes_{\mathbf{F}_1}\mathbf{Z}$.

\medskip
(iv)  Let $[S_{g,n}]$ be apalindromic,  i.e.   $p(x_1^{\mp 1},\dots, x_{6g-6+2n}^{\mp 1})\ne p(x_1^{\pm 1},\dots, x_{6g-6+2n}^{\pm 1})$. 
Likewise, the involution $\theta:  x_j^{\pm 1}\mapsto x_j^{\mp 1}$ pulls back to complex conjugation, but in this case 
there are no invarints of conjugation except for  $\mathbb{A}_{RM}^{6g-6+2n}$.
 We conclude therefore that  
$\exp ~[K_0^+(\mathscr{A}_{RM}^{6g-6+2n})]\otimes\mathbf{Q}$
is isomorphic to the function field 
\linebreak
$\mathbf{F}_{1}[x_1^{\pm 1},\dots, x^{\pm 1}_{6g-6+2n}]\otimes_{\mathbf{F}_1}\mathbf{Z} ~/~[S_{g,n}]\otimes_{\mathbf{F}_1}\mathbf{Z}$.

\bigskip
Lemma \ref{lm3.2} is proved. 
\end{proof}

\begin{lemma}\label{lm3.3}
An extension of degree $m\ge 1$ of $\mathbf{F}_1$ corresponds to adjoining the $m$-th root 
of unity (the real part of  $m$-th root 
of unity, resp.) to the field 
\linebreak
$\exp ~[K_0^+(\mathscr{A}_{RM}^{6g-6+2n})]\otimes\mathbf{Q}$
($\mathfrak{Re}~\exp ~[K_0^+(\mathscr{A}_{RM}^{6g-6+2n})]\otimes\mathbf{Q}$, resp.), 
i.e. 
\begin{equation} 
\begin{array}{lcr}
  \mathbf{F}_{1^m}[x_1^{\pm 1},\dots, x^{\pm 1}_{6g-6+2n}] \otimes_{\mathbf{F}_1}\mathbf{Z}~/~[S_{g,n}]\otimes_{\mathbf{F}_1}\mathbf{Z}
\cong&&\\
&&\\
\cong
\begin{cases} \mathbf{Q}\left(e^{2\pi i\alpha_k +\log\log\varepsilon} ~|~1\le k\le 6g-7+2n \right)
(e^{\frac{2\pi i}{m}}), & if ~[S_{g,n}] ~is\cr
& apalindromic,\cr
               \mathbf{Q}\left(\cos 2\pi(\alpha_k +\frac{1}{m})\times\log\varepsilon ~|~1\le k\le 6g-7+2n\right), 
               &  if ~[S_{g,n}] ~is\cr
& palindromic.
\nonumber
\end{cases}               
&&
\end{array}
\end{equation}
\end{lemma} 
\begin{proof}
(i) Let $f_m: \widetilde{\operatorname{Drin}}_A^{r}(k)\to \operatorname{Drin}_A^{r}(k)$
be an isogeny of degree $m\ge 1$ between  the Drinfeld modules $\widetilde{\operatorname{Drin}}_A^{r}(k)$ and $\operatorname{Drin}_A^{r}(k)$ [Rosen 2002] \cite[Section 12]{R}.  
The functor $F: \operatorname{Drin}_A^{3g-3+n}(k)\to \mathscr{A}_{RM}^{6g-6+2n}$ maps $f_m$ to a $C^*$-homomorphism $F(f_m): \widetilde{\mathscr{A}}_{RM}^{6g-6+2n}
\to \mathscr{A}_{RM}^{6g-6+2n}$  of the corresponding noncommutative tori $\widetilde{\mathscr{A}}_{RM}^{6g-6+2n}$ and $\mathscr{A}_{RM}^{6g-6+2n}$,
where $\{\widetilde{\alpha}_k=\alpha_k+\frac{1}{m} ~|~1\le k\le 6g-7+2n\}$  \cite[Theorem 3.3]{Nik1}.
In particular,  the image of  the torsion submodule $\widetilde{\Lambda}_{\rho}[a]$ is given by the formula:
\begin{equation}\label{eq3.2}
F(\widetilde{\Lambda}_{\rho}[a])=e^{2\pi i(\alpha_k+\frac{1}{m}) +\log\log\varepsilon},   \qquad 1\le k\le 6g-7+2n. 
\end{equation}

\medskip
(ii)  Let $[S_{g,n}]$ be apalindromic.  In view of Lemma \ref{lm3.2},  one gets 
an ismorphism of the fields 
 $\mathbf{F}_{1}[x_1^{\pm 1},\dots, x^{\pm 1}_{6g-6+2n}]\otimes_{\mathbf{F}_1}\mathbf{Z} ~/~[S_{g,n}]\otimes_{\mathbf{F}_1}\mathbf{Z}
 \cong \exp ~[K_0^+(\mathscr{A}_{RM}^{6g-6+2n})]\otimes\mathbf{Q}$. 
In view of the commutative diagram in Figure 1, the $m$-th order  extension  $\mathbf{F}_1\subseteq \mathbf{F}_{1^m}$ 
gives rise to a sugroup of  index $m$ of the Grothendieck semigroup $K_0^+(\mathscr{A}_{RM}^{6g-6+2n})$, i.e.
\begin{equation}\label{eq3.3}
\mathbf{F}_{1^m}[x_1^{\pm 1},\dots, x^{\pm 1}_{6g-6+2n}]\otimes_{\mathbf{F}_1}\mathbf{Z} ~/~[S_{g,n}]\otimes_{\mathbf{F}_1}\mathbf{Z}
 \cong \exp ~[K_0^+(\widetilde{\mathscr{A}}_{RM}^{6g-6+2n})]\otimes\mathbf{Q},
\end{equation}
where $\widetilde{\mathscr{A}}_{RM}^{6g-6+2n}$ comes from the isogeny $f_m$ between the Drinfeld modules  
$\widetilde{\operatorname{Drin}}_A^{r}(k)$ and $\operatorname{Drin}_A^{r}(k)$. 
It remains to combine formulas (\ref{eq3.2}) and (\ref{eq3.3}) with the obvious 
identity $\exp  \left(2\pi i(\alpha_k+\frac{1}{m}) +\log\log\varepsilon\right)= \exp  \left(2\pi i\alpha_k +\log\log\varepsilon\right)\exp \left( \frac{2\pi i}{m}\right)$.
The conclusion of Lemma \ref{lm3.3} for the $[S_{g,n}]$  apalindromic is proved.   

\medskip
(iii)  Let $[S_{g,n}]$ be palindromic.  By  Lemma \ref{lm3.2},  one gets 
 $\mathbf{F}_{1}[x_1^{\pm 1},\dots, x^{\pm 1}_{6g-6+2n}]\otimes_{\mathbf{F}_1}\mathbf{Z} ~/~[S_{g,n}]\otimes_{\mathbf{F}_1}\mathbf{Z}
 \cong \mathfrak{Re} ~\exp ~[K_0^+(\mathscr{A}_{RM}^{6g-6+2n})]\otimes\mathbf{Q}$. 
We repeat the argument of item (ii) and notice that $\mathfrak{Re}  ~\exp  \left(2\pi i(\alpha_k+\frac{1}{m}) +\log\log\varepsilon\right)=\cos 2\pi(\alpha_k +\frac{1}{m})\times\log\varepsilon$. 
The conclusion of Lemma \ref{lm3.3} when the $[S_{g,n}]$  palindromic follows.

\bigskip
Lemma \ref{lm3.3} is proved. 
\end{proof}

\bigskip
Theorem \ref{thm1.2} follows from Lemma \ref{lm3.3}.

\subsection{Proof of Corollary \ref{cor1.5}}
(i) Recall that the Galois groups  
\linebreak
 $\operatorname{Gal} \left(\mathbf{Q}(e^{2\pi i\alpha_k+\log\log\varepsilon} ~|~1\le k\le 6g-7+2n)  ~| ~\mathbf{Q}\right)\subseteq \operatorname{GL}_{3g-3+n}\left(A/aA\right)$
 and
 \linebreak
$\operatorname{Gal} \left(\mathbf{Q}(cos ~2\pi\alpha_k \times\log\varepsilon ) ~|~1\le k\le 6g-7+2n)  ~| ~\mathbf{Q}\right)\subseteq \operatorname{GL}_{3g-3+n}\left(A/aA\right)$.
In particular, whenever the rank of the Drinfeld module $r=3g-3+n=1$, one gets the abelian Galois group,  since $A/aA$ is a commutative ring.  

\medskip
(ii) On the other hand, the topological type $(g,n)$ satisfying  condition  $3g-3+n=1$ with  $g\ge 0$ and $n\in\{1, 2, 3\}$ has the unique 
solution $g=n=1$.  Thus any number field of the topological type $(1,1)$, i.e.  once-punctured torus, is an abelian extension of the field $\mathbf{Q}$.

\medskip
(iii) It remains to compare Theorem \ref{thm1.2} for $g=n=1$ with the results  \cite[Corollary 3.4]{Nik1}.  

\bigskip
Corollary \ref{cor1.5} is proved.

\section{Number fields of topological type $(1,1)$}
We conclude by classification of the number fields of topological type $(1,1)$
correponding to the once-punctured torus. Corollary \ref{cor1.5} implies the  table in Figure 2.

\begin{remark}
The degenerate case $\alpha\in\mathbf{Q}$ implies $\alpha=\varepsilon=1$ in formulas (\ref{eq1.2}),
so that $\mathbf{K}\cong\mathbf{Q}(0, e^{\frac{2\pi i}{m}})\cong\mathbf{Q}(e^{\frac{2\pi i}{m}})$ is the cyclotomic
field; hence the first row in Figure 2. 
Likewise, the  transcendental functions  $e^{2\pi i\alpha +\log\log\varepsilon}$ ($\cos 2\pi\alpha \times\log\varepsilon$, resp.)
is known to generate the Hilbert class field (maximal abelian unramified extension, resp.)  of the imaginary (real, resp.) quadratic field \cite[Introduction]{Nik1};
  hence the second and third rows in Figure 2. 
\end{remark}

\begin{figure}
\begin{tabular}{|c|c|c|c|}
\hline
&&&\\
$\alpha\in\mathbf{R}$ & $m\in\mathbf{Z}$ & Number field & Theorem\\
&&&\\
\hline
rational & $m\ge 1$ & cyclotomic & Kronecker-Weber\\
\hline
quadratic irrational & $m=1$ & Hilbert class field & complex multiplication\\
  case  $\mathbf{K}\subset\mathbf{C}-\mathbf{R}$    &             & of imaginary quadratic field &\\
\hline
quadratic irrational & $m=1$ & maximal abelian extension & real multiplication\\
  case   $\mathbf{K}\subset\mathbf{R}$                           &             & of real quadratic field &\\
\hline
\end{tabular}

\caption{Number fields of topological type $(1,1)$.}
\end{figure}


\section*{Data availability}
  
  Data sharing not applicable to this article as no datasets were generated or analyzed during the current study.
   
\section*{Conflict of interest}
On behalf of all co-authors, the corresponding author states that there is no conflict of interest.
  

\section*{Funding declaration}
The author was partly supported by the NSF-CBMS grant 2430454.

\bibliographystyle{amsplain}


\end{document}